\documentclass[11pt]{amsart}
\usepackage[perpage,symbol*]{footmisc}
\usepackage{amsopn}
\usepackage{amssymb, amscd}
\usepackage{color}
\usepackage [all] {xy}
\usepackage{hyperref}
\usepackage{enumerate}
\usepackage{array}
\usepackage{multirow}
\usepackage{graphicx}
\usepackage{color}

\newcommand{\nc}{\newcommand}

\nc{\fg}{\mathfrak{f} }     \nc{\vg}{\mathfrak{v} }       \nc{\wg}{\mathfrak{w} }
\nc{\zg}{\mathfrak{z} }     \nc{\ngo}{\mathfrak{n} }      \nc{\kg}{\mathfrak{k} }
\nc{\mg}{\mathfrak{m} }     \nc{\bg}{\mathfrak{b} }       \nc{\ggo}{\mathfrak{g} }
\nc{\sog}{\mathfrak{so} }
\nc{\sug}{\mathfrak{su} }   \nc{\spg}{\mathfrak{sp} }     \nc{\slg}{\mathfrak{sl} }
\nc{\glg}{\mathfrak{gl} }   \nc{\cg}{\mathfrak{c} }       \nc{\rg}{\mathfrak{r} }
\nc{\hg}{\mathfrak{h} }     \nc{\tg}{\mathfrak{t} }       \nc{\ug}{\mathfrak{u} }
\nc{\dg}{\mathfrak{d} }     \nc{\ag}{\mathfrak{a} }       \nc{\pg}{\mathfrak{p} }
\nc{\sg}{\mathfrak{s} }     \nc{\affg}{\mathfrak{aff} }
\nc{\ggob}{\overline{\mathfrak{g}} }

\nc{\pca}{\mathcal{P}}       \nc{\nca}{\mathcal{N}}       \nc{\lca}{\mathcal{L}}
\nc{\oca}{\mathcal{O}}       \nc{\mca}{\mathcal{M}}       \nc{\tca}{\mathcal{T}}
\nc{\aca}{\mathcal{A}}       \nc{\cca}{\mathcal{C}}       \nc{\gca}{\mathcal{G}}
\nc{\sca}{\mathcal{S}}       \nc{\hca}{\mathcal{H}}       \nc{\bca}{\mathcal{B}}
\nc{\dca}{\mathcal{D}}       \nc{\rca}{\mathcal{R}}

\nc{\val}{\operatorname{val}}

\nc{\vp}{\varphi}
\nc{\ddt}{\tfrac{{\rm d}}{{\rm d}t}}
\nc{\dpar}{\tfrac{\partial}{\partial t}}

\nc{\im}{\mathtt{i}}                 

\newcommand{\fna}[5]{
   \begin{array}{rccl}
       #1: & \hspace{-2mm} #2 &\hspace{-2mm} \longrightarrow &\hspace{-2mm} {#3}          \\
        & \hspace{-2mm} #4 &\hspace{-2mm} \longmapsto     &\hspace{-2mm} {#5}
    \end{array}}

\nc{\SO}{\mathrm{SO}}           \nc{\Spe}{\mathrm{Sp}}          \nc{\Sl}{\mathrm{SL}}
\nc{\SU}{\mathrm{SU}}           \nc{\U}{\mathrm{U}}
\nc{\Se}{\mathrm{S}}            \nc{\Cl}{\mathrm{Cl}}           \nc{\Spein}{\mathrm{Spin}}
\nc{\Pin}{\mathrm{Pin}}
\nc{\Or}{\mathrm{O}_n(\RR)}
\nc{\Glr}{\mathrm{GL}_n(\RR)}   \nc{\Glc}{\mathrm{GL}_n(\CC)}   \nc{\Glv}{\mathrm{GL}(V)}    \nc{\Glk}{\mathrm{GL}_n(\fk)}   \nc{\Gl}{\mathrm{GL}}
\nc{\Grp}{\mathrm{G}}

\nc{\g}{\mathfrak{gl}_n(\RR)}

\nc{\RR}{{\Bbb R}} \nc{\HH}{{\Bbb H}} \nc{\CC}{{\Bbb C}} \nc{\ZZ}{{\Bbb Z}}
\nc{\FF}{{\Bbb F}} \nc{\NN}{{\Bbb N}} \nc{\QQ}{{\Bbb Q}} \nc{\PP}{{\Bbb P}}

\nc{\euler}{{\rm e}}

\nc{\vs}{\vspace{.2cm}} \nc{\vsp}{\vspace{1cm}}

\nc{\ip}{\langle \cdot , \cdot \rangle}
\nc{\ipp}{(      \cdot , \cdot       )}
\nc{\la}{\langle} \nc{\ra}{\rangle}

\nc{\ortsum}{ \mbox{\tiny $\displaystyle \bigoplus^{\perp}$}}

\nc{\unm}{\tfrac{1}{2}}\nc{\unc}{\tfrac{1}{4}} \nc{\und}{\tfrac{1}{16}}

\nc{\no}{\vs\noindent}

\nc{\lamn}{\Lambda^2(\RR^n)^*\otimes\RR^n}        \nc{\lamnc}{\Lambda^2(\CC^n)^*\otimes\CC^n}
\nc{\lamp}{\Lambda^2\pg^*\otimes\pg}
\nc{\lamg}{\Lambda^2\ggo^*\otimes\ggo}            \nc{\lamngo}{\Lambda^2\ngo^*\otimes\ngo}
\nc{\lamnk}{\Lambda^2(\fk^n)^*\otimes \fk^n}      \nc{\lamnkt}{\Lambda^3(\fk^n)^*\otimes \fk^n}

\nc{\tangz}{{\rm T}^{\rm Zar}}
\nc{\mum}{/\!\!/} \nc{\kir}{/\!\!/\!\!/}
\nc{\lievark}{\mathfrak{L}_n(\fk)}         \nc{\lievarc}{\mathfrak{L}_n(\CC)}        \nc{\lievarr}{\mathfrak{L}_n(\RR)}
\nc{\solvvark}{\mathfrak{R}_n(\fk)}        \nc{\solvvarc}{\mathfrak{R}_n(\CC)}       \nc{\solvvarr}{\mathfrak{R}_n(\RR)}
\nc{\nilvark}{\mathfrak{N}_n(\fk)}         \nc{\nilvarc}{\mathfrak{N}_n(\CC)}        \nc{\nilvarr}{\mathfrak{N}_n(\RR)}
\nc{\cirre}{\textrm{C}}
\nc{\fk}{\mathrm{k}}

\nc{\Ri}{\tfrac{4\Ric_{\mu}}{||\mu||^2}}

\nc{\ds}{\displaystyle}

\nc{\ben}{\begin{enumerate}} \nc{\een}{\end{enumerate}}

\nc{\f}{\frac}

\nc{\lb}{[\cdot,\cdot]}

\nc{\isn}{\tfrac{1}{||v||^2}}

\nc{\gkp}{(\ggo=\kg\oplus\pg,\ip)} \nc{\ukh}{(\ug=\kg\oplus\hg,\ip)}

\nc{\Hess}{\operatorname{Hess}}
\nc{\diag}{\operatorname{Diag}}
\nc{\ad}{\operatorname{ad}}       \nc{\Ad}{\operatorname{Ad}}        
\nc{\rank}{\operatorname{rank}}   \nc{\codim}{\operatorname{codim}}  
\nc{\Irr}{\operatorname{Irr}}     \nc{\End}{\operatorname{End}}
\nc{\Aut}{\operatorname{Aut}}     \nc{\Inn}{\operatorname{Inn}}
\nc{\lRad}{\operatorname{Rad}}
\nc{\Der}{\operatorname{Der}}     \nc{\Ker}{\operatorname{Ker}}      \nc{\spanv}{\operatorname{span}}
\nc{\Iso}{\operatorname{I}}       \nc{\Diff}{\operatorname{Diff}}
\nc{\Lie}{\operatorname{Lie}}     \nc{\tr}{\operatorname{tr}}
\nc{\dif}{\operatorname{d}}
\nc{\sen}{\operatorname{sen}}     \nc{\tang}{\operatorname{T}}
\nc{\modu}{\operatorname{mod}}
\nc{\Riem}{\operatorname{Rm}}     \nc{\Ric}{\operatorname{Ric}}
\nc{\sym}{\operatorname{sym}}     \nc{\symac}{\operatorname{sym^{ac}}}   \nc{\symc}{\operatorname{sym^{c}}}
\nc{\scalar}{\operatorname{sc}}
\nc{\grad}{\operatorname{grad}}
\nc{\ricci}{\operatorname{ric}}   \nc{\nr}{\operatorname{nr}}            \nc{\riccic}{\operatorname{ric^{c}}}
\nc{\riccig}{\operatorname{ric^{\gamma}}}
\nc{\Rin}{\operatorname{M}}
\nc{\Le}{\operatorname{L}}
\nc{\level}{\operatorname{level}} \nc{\rad}{\operatorname{r}}
\nc{\abel}{\operatorname{ab}} \nc{\CH}{\operatorname{CH}}
\nc{\mcc}{\operatorname{mcc}} \nc{\Adj}{\operatorname{Adj}}
\nc{\Order}{\operatorname{O}} \nc{\Ricg}{\operatorname{Ric^{\gamma}}}
\nc{\Hom}{\operatorname{Hom}}
\nc{\sign}{\operatorname{sign}}

\nc{\rhov}{\operatorname{\rho_{v}}}
\nc{\mm}{m}
\nc{\mmt}{\widetilde{m}}
\nc{\F}{\operatorname{F}}

\theoremstyle{plain}
\newtheorem{theorem}{Theorem}[section]

\newtheorem{corollary}[theorem]{Corollary}

\theoremstyle{definition}
\newtheorem{definition}[theorem]{Definition}

\theoremstyle{remark}

\newtheorem{example}[theorem]{Example}

\numberwithin{equation}{section}

\begin{document}

\title{Classification of 7-dimensional Einstein Nilradicals}

\author{EDISON ALBERTO FERN\'ANDEZ CULMA}

\address{Current affiliation: CIEM, FaMAF, Universidad Nacional de C\'ordoba, \newline \indent Ciudad Universitaria, \newline \indent (5000) C\'ordoba, \newline \indent Argentina}

\email{efernandez@famaf.unc.edu.ar}

\thanks{Fully supported by a CONICET fellowship (Argentina)}



\begin{abstract}
The problem of classifying Einstein solvmanifolds, or equivalently, Ricci soliton nilmanifolds, is known to be equivalent to a question on the variety $\mathfrak{N}_n(\CC)$ of $n$-dimensional complex nilpotent Lie algebra laws. Namely, one has to determine which $\mathrm{GL}_n(\CC)$-orbits in $\mathfrak{N}_n(\CC)$ have a critical point of the squared norm of the moment map. In this paper, we give a classification result of
such distinguished orbits for $n = 7$.  The set $\mathfrak{N}_7(\CC) / \mathrm{GL}_7(\CC)$ is formed by $148$ nilpotent Lie algebras and $6$ one-parameter families of pairwise non-isomorphic nilpotent Lie algebras. We have applied to each Lie algebra one of three main techniques to decide whether it has a distinguished orbit or not.
\end{abstract}

\maketitle

\section{Introduction}

From general theory of relativity, a Riemannian manifold $(M, g)$ is said to be {\it Einstein} if its Ricci tensor complies the Einstein condition $\ricci=c g$ for some constant $c\in \RR$ and in such case, $g$ is called an {\it Einstein metric}. If $M$ is compact, Einstein metrics are identified as critical points of the total scalar curvature functional on the space of all metrics on $M$ of a fixed volume, and can therefore be considered as \textquotedblleft{selected elements}\textquotedblright.

At the present time, no general existence results for Einstein metrics are known. Even in the homogeneous case, $M=G/H$, the search for $G$-invariant Einstein metrics ($G$-invariance turns the problem into an \textquotedblleft{algebraic}\textquotedblright one) is a difficult problem. We are interested in homogeneous Einstein metrics of negative scalar curvature (noncompact, nonflat). A big challenge in this case is known as the {\it Alekseevski\v{\i} conjecture} (\cite[Conjecture 7.57]{BESSE1}), which essentially states that any homogeneous Einstein manifold with negative scalar curvature is isometric to a simply connected solvable Lie group endowed with a left-invariant metric satisfying the Einstein condition (commonly so-called \textit{Einstein solvmanifolds}).

In \cite{HEBER1}, Heber gives a detailed analysis of \textit{standard Einstein solvmanifolds}, this means that the corresponding metric solvable Lie algebra $(\sg, \ip)$ has orthogonal decomposition $\sg=\ngo \, \ortsum \, \ag$ with $\ngo={[\sg,\sg]}$ and ${[\ag,\ag]}=0$. Heber's paper is accepted by many researchers as a pioneer work and one of his results allows Einstein solvmanifolds to be studied by using geometric invariant theory (GIT). Since then, geometric invariant theory has played an important role in understanding Einstein solvmanifolds and recent advances have come by this powerful machinery. A sample of this is the next result due to Lauret, in which the Kirwan's stratification of the variety of nilpotent Lie algebra laws is in the core of the proof.

\begin{theorem}\label{LauretStandard}\cite[Theorem 3.1]{LAURET4}
Any Einstein solvmanifold is standard.
\end{theorem}

It is well known that any Einstein solvmanifolds is determined by the nilradical of its Lie algebra, which is so-called an \textit{Einstein nilradical}. In \cite{NIKOLAYEVSKY2}, Nikolayevsky gives some structural results on Einstein nilradicals, as well as criteria to decide if a nilpotent Lie algebra is Einstein nilradical or not.

There are many equivalent conditions to be an Einstein nilradical, some of them related with the Ricci flow. Let $N$ be the simply connected nilpotent Lie group with Lie algebra $\ngo$. Then $(S,\ip)$ is Einstein if and only if $(N,\ip|_{\ngo})$ is a {\it Ricci soliton} (called nilsolitons in the literature, cf. \cite{LAURET3}), i.e. the Ricci flow solution starting at $(N,\ip|_{\ngo})$ only evolves by pullback of diffeomorphisms and scaling. Furthermore, the only known examples of nontrivial homogeneous Ricci solitons are all solvmanifolds, which can be constructed as a semidirect product of an Einstein nilradical and a suitable torus of automorphisms (see \cite{LAURET5}).

Our main result in this paper is a complete classification of $7$-dimensional nilpotent Lie algebras which are Einstein nilradicals. Since any nilpotent Lie algebra of dimension less or equal than $6$ is an Einstein nilradical (\cite{LAURET2}, \cite{WILL1}) and a direct sum of Einstein nilradicals is again an Einstein nilradical, we focus on studying indecomposable algebras. By \cite[Theorem 6]{NIKOLAYEVSKY2}, it is enough to consider complex nilpotent Lie algebras and this is why we use the classification in dimension $7$ given in \cite{CARLES1} (with a couple of corrections by Magnin in \cite{MAGNIN1}), which consists of a list of $117$ non-isomorphic indecomposable algebras and $6$ one-parameter families of pairwise non-isomorphic nilpotent Lie algebras. If a nilpotent Lie algebra is written in a \textit{nice basis} (see Definition \ref{nicebasis}), then by using \cite[Theorem 3]{NIKOLAYEVSKY2}\label{nikonicebasis} it is a simple matter to prove whether the algebra is Einstein nilradical or not.  In \cite{MAGNIN1}, $20$ algebras and a one-parameter family of $\rank$ $\geq 1$ are not written in a nice basis, so that these are studied by other methods. Namely,
\begin{enumerate}[\hspace*{0cm} ${\cdot}$]
   \item by exhibiting a nilsoliton inner product (see Corollary \ref{beta}). \\
   \item by studying the closedness of the orbit of a nilpotent Lie algebra law by the action of certain reductive algebraic subgroup of $\Glr$ on $\lamn$ (cf. \cite[Theorem 5]{NIKOLAYEVSKY2}).
\end{enumerate}
These three main techniques have been applied to each $7$-dimensional algebra in \cite{MAGNIN1} of rank $\geq 1$ and the full classification of $7$-dimensional Einstein nilradicals obtained is given in Tables 1,2,3 and 4, according to their rank. In Section 3, many examples are given in order to illustrate the application of the techniques, the rest of the algebras can be each worked out in a completely analogous and straightforward way. We refer to \cite{FERNANDEZ-CULMA2} for more information on each algebra, including computations, solutions to equations and isomorphisms that we have used to obtain the complete classification (see also \cite{FERNANDEZ-CULMA1}).  It is important to point out that the present paper contains all the required arguments to get the classification theorem.

The first $\rank$-zero nilpotent Lie algebras (also known as characteristically nilpotent Lie algebras) appear in dimension $7$, which can not be Einstein nilradicals since they do not admit an $\NN$-gradation. This family has $7$ nilpotent Lie algebras and a one-parameter family. For $\rank$ $\geq 1$, we get $82$ indecomposable Einstein nilradicals (from a total of $110$ indecomposable algebras) and $5$ one-parameter families of Einstein nilradicals (with the exception of at most $2$ points in each one-parameter family).

After finishing the present paper, we have become aware that Payne and Kadio\u{g}lu have introduced in \cite{PAYNE1} a computational method for classifying Einstein nilradicals in the family of nilpotent Lie algebras with simple pre-Einstein derivation (i.e. all eigenvalues are different) and nonsingular Gram matrix (a simple derivation implies the existence of a nice basis which makes sense to speak of the Gram matrix). Their method does not rely on any preexisting classifications of nilpotent Lie algebras. They classify Einstein nilradicals within this family, which has $33$ algebras in dimensions $7$ and $159$ algebras in dimension $8$. With the notation in \cite{MAGNIN1}, the $7$-dimensional indecomposable algebras that have been studied in \cite{PAYNE1} are: $1.8$, $1.9$, $2.42$ and $2.29$ in \cite[Table I]{PAYNE1}, and $1.6$, $2.3$, $1.7$, $1.10$, $2.32$, $2.7$, $2.28$, $3.15$, $3.6$, $2.39$, $2.22$, $2.20$, $2.16$, $2.45$, $3.21$, $2.21$, $2.40$, $2.36$, $2.33$, $2.43$, $2.29$, $2.44$, $2.34$, $2.18$ in \cite[Table III]{PAYNE1} respectively.

It is known that the variety of $n$-dimensional complex nilpotent Lie algebra laws, $\nilvarc$, is reducible for any $n\geq 7$ (cf. \cite{KHAKIMDJANOV1}). For the natural action of $\Glc$ on $\lamnc$, by applying Kirwan's stratification and convexity properties of the moment map, it follows that each irreducible component $\mathcal{C}_i$ of $\nilvarc$ can be obtained as the closure of its unique stratum of minimum norm (cf. \cite{FERNANDEZ-CULMA2}). There are finitely many strata and are parameterized by the moment map images of the critical points of the squared norm of the moment map that belong to $\nilvarc$ (which are precisely Einstein nilradicals). It follows from the classification of $7$-dimensional Einstein nilradicals obtained in the present paper that the two irreducible components of $\mathfrak{N}_7(\CC)$ are respectively the closures of the strata corresponding to the Einstein nilradicals of eigenvalue type $(1<2<3<4<5; 2,1,2,1,1)$ (Table 1, algebra $17$) and $(1<2<3<4<5<6<7;1,...,1)$ (Table 1, algebras $1(i_\lambda)$ with $\lambda\neq 0,1$ and $1(iii)$, and Table 2, algebras $5$ and $14$) (see definition \ref{einNIL}). We note that the stratum associated to the first type above has actually minimal norm over all strata of $\mathfrak{N}_7(\CC)$, although a third stratum has smaller norm than the second one, the strata attached to the type $(16<21<37<48<53<69<90;1,...,1)$ (Table 2, algebra $13$). A characterization of the strata of minimum norm in the irreducible components of $\mathfrak{N}_n(\CC)$ for $n$ large, may be very useful in the study of many problems on the variety of Lie algebra laws, as for instance rigidity of nilpotent Lie algebras (a conjecture attributed to Mich{\`e}le Vergne says that for $n$ sufficiently large, there do not exist nilpotent Lie algebras which have Zariski-open orbit in $\nilvarc$).

\section{Preliminaries}

In this section, we overview all the results on Einstein nilradicals we need to apply in our classification. We refer to the survey \cite{LAURET3} for a more detailed treatment.

It follows from Theorem \ref{LauretStandard} and Heber's Rank-One reduction (\cite[Theorem 4.18]{HEBER1}) that any Einstein solvmanifold with Einstein metric solvable Lie algebra $(\sg=\ngo\oplus\ag , \ip)$ is determined by the Einstein metric solvable Lie algebra $(\ngo \oplus \ag_{1}, \la\la \cdot , \cdot \ra\ra)$ where $\la\la \cdot , \cdot \ra\ra$ denotes the restriction of the inner product on $\sg$ and $\ag_{1}$ is a special one-dimensional subspace of $\ag$. Since $\ag_{1}$ and $\la\la \cdot , \cdot \ra\ra$ only depend on the nilpotent Lie algebra $\ngo$ (see \cite[Lemma 2]{LAURET1}), the classification of Einstein solvmanifolds is actually a problem on nilpotent Lie algebras.

\begin{definition}\cite[Definition 1]{NIKOLAYEVSKY2}\label{einNIL}
A nilpotent Lie algebra which is the nilradical of an Einstein metric solvable Lie algebra is called an {\it Einstein nilradical}. A vector $X$ spanning $\ag_{1}$ and scaled in such a way that $||X||^2=\tr(\ad(X))$ is called the \textit{mean curvature vector}, and the restriction of the derivation $\ad(X)$ to $\ngo$ the \textit{Einstein derivation}. The ordered set of eigenvalues $d_i \in \NN$ of the (appropriately scaled) Einstein derivation, together with the multiplicities $n_i$ is called the \textit{eigenvalue type} of an Einstein solvmanifold,  which is written as $(d_1 <\ldots  < d_r; n_1,\ldots , n_r)$ (see \cite{HEBER1})
\end{definition}

Let $V$ be the vector space $\lamn$. If $\{e_1^{'},...,e_n^{'}\}$ is the basis of $(\RR^n)^*$ dual to the canonical basis $\{e_1,...,e_n\}$ of $\RR^n$ then $\mu_{ij}^k=(e_i^{'}\wedge e_j^{'})\otimes e_k$ with $1\leq i < j \leq n$, $1\leq k \leq n$ is the canonical basis of $V$. As we mentioned above, GIT is a very important tool in the theory of Einstein solvmanifolds and the connection is the moment map for the natural action of $\Glr$ on $V$, $\mmt: V \smallsetminus\{ 0\} \longrightarrow \sym(n)$. Let $\mu \in V$ be a nilpotent Lie algebra law and let $\Ric_{\mu}$ denote the Ricci operator of the nilmanifold $(N_{\mu},\ip)$, where $N_{\mu}$ is the simply connected Lie group with $\Lie(N_{\mu})=(\RR^n,\mu)$ and $\ip$ is the canonical inner product of $\RR^n$. For those who are familiar with the works of Ness (\cite{NESS1}), Kirwan (\cite{KIRWAN1}) or Marian (\cite{MARIAN1}), it will be enough to know that
\begin{equation}\label{momentmapricci}
\mmt(\mu)=\frac{4}{||\mu||^2}\Ric_{\mu},
\end{equation}

\noindent to recognize some of the results about existence and uniqueness of Einstein solvmanifolds (Einstein nilradical).

Let $\mm$ be the unnormalized moment map, given by $\mm(\mu)=4\Ric_{\mu}$. By analogy with GIT, we work with the moment map $\mm$ instead of $\Ric$. Computations in \cite{FERNANDEZ-CULMA2, FERNANDEZ-CULMA1} are made using $\mm$.

The next theorem gives a characterization of Einstein nilradicals.

\begin{theorem}\cite[Theorem 4.2.]{LAURET3}\label{characterization}
Let $\mu \in V$ be a nilpotent Lie algebra law. The nilpotent Lie algebra $(\RR^n,\mu)$ is an Einstein nilradical if and only if there exists $\widetilde{\mu} \in \Glr\cdot\mu$ such that $\Ric_{\widetilde{\mu}} \in \RR I \oplus \Der(\widetilde{\mu})$, or equivalently, $\mm ( \widetilde{\mu} ) = c_{ \widetilde{\mu} }I + \phi$ for some $c_{ \widetilde{\mu} } < 0$ and $\phi$ a symmetric derivation of $\widetilde{\mu}$.
\end{theorem}

The derivation $\phi$ in Theorem \ref{characterization} is, up to conjugation, a positive multiple of the Einstein derivation. Recall that the eigenvalues of the Einstein derivation are positive integer and define the eigenvalue type (see \cite{LAURET3}).

\begin{corollary}\label{beta}
Let $\mu \in V$ be a nilpotent Lie algebra law. The nilpotent Lie algebra $(\RR^n,\mu)$ is an Einstein nilradical with eigenvalue type $(d_1 < ... < d_r ; n_1,..,n_r)$ if and only if there exists $\widetilde{\mu} \in \Glr \cdot \mu$ such that:
\begin{equation}\label{betaequ}
\mm ( \widetilde{\mu} ) = \tfrac{\sum n_i d_i}{n\sum n_i d_i^2 - \left(\sum n_i d_i\right)^2} \left( - \tfrac{\sum n_i d_i^2 }{\sum n_i d_i}\small{\mathrm{I}} +
{\diag ({\underbrace{d_1}_{n_1\,times}},...,\underbrace{d_r}_{{n_r\,times}})}\right)
\end{equation}

\noindent where $d_i$, $i=1,...,r$ are positive integers without a common divisor and $\diag (d_1,...,d_r)$ is a derivation of $\widetilde{\mu}$.
\end{corollary}

\begin{proof}
Let $\mu_0 \in \Glr \cdot \mu$ such that $\mm(\mu_0)=c_{\mu_{0}}I+\phi$. Since the matrix of $\mm(\mu_0)$ in the canonical basis of $\RR^n$ is symmetric, then there exists $k\in\Or$ such that $k\mm(\mu_0)k^{-1}=\diag(x_1,...,x_n)$ with $x_1\leq...\leq x_n$ and so $k\phi k^{-1} = a\diag(d_1,...,d_r)$, where $d_1<...<d_r$ are the eigenvalues of the Einstein derivation $\phi$ with multiplicities $n_1,...,n_r$. By using that the moment map is $\Or-$equivariant we have
\begin{equation}\label{st1}
  \mm(k\cdot\mu_0)=c_{\mu_0} I + a\diag(d_1,...,d_r).
\end{equation}
Now, since the moment map is orthogonal to each symmetric derivation (by \cite[Equation (4.4)]{LAURET3}), if we take the inner product of equation (\ref{st1}) with $\diag(d_1,...,d_r)$, then we get that
$$
  a= -c_{\mu_{0}}\frac{\sum n_i d_i}{\sum n_i d_i^2}
$$
and therefore, by replacing $a$ in (\ref{st1})
\begin{equation}\label{st2}
  \mm(k\cdot\mu_0)=c_{\mu_0}\left(I - \frac{\sum n_i d_i}{\sum n_i d_i^2} \diag(d_1,...,d_r)\right)
\end{equation}
We recall that, for any $\lambda \in V$, $\tr(\mm(\lambda))=-||\lambda||^2$ (again, by \cite[Equation 4.4.]{LAURET3} with $\alpha=I$), therefore, if we use this fact in (\ref{st2}), it follows that
$$
c_{\mu_0}=-||c_{\mu_0}||^2/\left({n-\frac{\left(\sum n_i d_i\right)^2}{\sum n_i d_i^2}}\right)
$$
the denominator above is not zero by the Cauchy–Schwarz inequality and $\phi \notin \RR I$. In consequence
$$
\frac{\mm(k\cdot\mu_0)}{||k\cdot\mu_0||^2}=\mm\left(\frac{k\cdot\mu_0}{||k\cdot\mu_0||}\right)= (\ref{betaequ})
$$
and, as $\frac{k\cdot\mu_0}{||k\cdot\mu_0||} \in \Glr\cdot\mu$, making $\widetilde{\mu} = \frac{k\cdot\mu_0}{||k\cdot\mu_0||}$ we have concluded the proof.
\end{proof}

The importance of expression (\ref{betaequ}) is that it only depends on the eigenvalue type. We use the pre-Einstein derivation and Corollary \ref{beta} to find Einstein nilradicals of a fixed eigenvalue type.

If $(\RR^n, \mu)$ is an Einstein nilradical with eigenvalue type $(d_1 <...< d_r ; n_1,...,n_r)$, then the value
\begin{equation}\label{betanorm}
\left(n-\frac{\left(\sum n_i d_i\right)^2 }{\sum  n_i d_i^2}\right)^{-1}
\end{equation}
is the minimum value on $\Glr \cdot \mu$ that the function
\begin{equation}\label{momentnorm}
\fna{\F}{V \smallsetminus\{ 0\}}{\RR}{v}{||\mmt(v)||^2}
\end{equation}
takes. $\F$ measures how far the Einstein derivation is from a multiple of the identity. Since there are finitely many eigenvalue types (or equivalently, there are finitely many \textit{strata}), then the expression (\ref{betanorm}) takes finitely many values which can be used to study degenerations in $\nilvarr$ via the stratification (\cite[Theorem 2.10.]{LAURET4}).

To finish, we review some definitions and results from \cite{NIKOLAYEVSKY2}.

\begin{definition}\cite[Definition 2]{NIKOLAYEVSKY2}
A derivation $\phi$ of a Lie algebra $\ggo$ is called \textit{pre-Einstein}, if it is semisimple, with all the eigenvalues real, and
\begin{equation}\label{preeinstein}
  \tr(\phi \psi) = \tr(\psi), \mbox{ for any } \psi\in \Der(\ggo)
\end{equation}
\end{definition}

\begin{theorem}\cite[Theorem 1]{NIKOLAYEVSKY2}\label{nikopreeinstein}
\begin{enumerate}[\hspace*{0cm} 1.]
\item
\begin{enumerate}[(a)]
    \item Any Lie algebra $\ggo$ admits a pre-Einstein derivation $\phi_{\ggo}$.
    \item The derivation $\phi_{\ggo}$ is determined uniquely up to automorphism of $\ggo$.
    \item All the eigenvalues of $\phi_{\ggo}$ are rational numbers.
  \end{enumerate}
\item Let $\ngo$ be a nilpotent Lie algebra, with $\phi$ a pre-Einstein derivation. If $\ngo$ is an Einstein nilradical, then its Einstein derivation is, up to conjugation by an automorphism, positively proportional to $\phi$ and
        \begin{equation}
          \phi > 0 \mbox{   and   } \ad_{\phi} \geq 0,
        \end{equation}
      i.e. all eigenvalues of $\phi$ are positive and all eigenvalues of $\ad_{\phi}$ are non negative.
\end{enumerate}
\end{theorem}

There is a reductive real Lie group, $G_{\phi}$, attached to a pre-Einstein derivation. Consider
\begin{equation}\label{gpreeinstein}
\ggo_{\phi}:=\{\alpha \in \glg_n(\RR) : {[\alpha,\phi]}=0, \, \tr(\alpha \phi)=0, \, \tr(\alpha)=0 \}
\end{equation}
and let $G_{\phi}$ be the connected Lie subgroup of $\Glr$ with Lie algebra $\ggo_{\phi}$. The group $G_{\phi}$ is in fact the identity component of the real algebraic group ${\widetilde{G_{\phi}}}$ which is given by
$$
{\widetilde{G_{\phi}}}:=\{ \diag(g_1,\ldots,g_r) / g_i \in \mathrm{GL}_{n_i}(\RR),\, \prod \det(g_i)=\prod \det(g_i)^{d_i}=1\}
$$
where $\phi= a \diag(\underbrace{d_1}_{n_1\mbox{ times}}, \ldots,\underbrace{d_r}_{n_r\mbox{ times}})$ with $a$
been the least common denominator of the eigenvalues of the pre-Einstein derivation $\phi$.

The relevance of the group $G_{\phi}$ in the study of Einstein nilradicals is given by the next theorem.
\begin{theorem}\cite[Theorem 5]{NIKOLAYEVSKY2}\label{nikoclosedorbit}
Let $\mu$ be a nilpotent Lie algebra law. For the nilpotent Lie algebra $\ngo=(\RR^n,\mu)$ with a pre-Einstein derivation $\phi$, the following conditions are equivalent:
\begin{enumerate}[\hspace*{0cm} (i)]
  \item $\ngo$ is an Einstein nilradical.
  \item The orbit $G_{\phi} \cdot \mu$ is closed in $V$.
  \item The function $\rho_{\mu}:G_{\phi} \longrightarrow \RR$ defined by $\rho_{\mu}(g)=||g \cdot \mu||^2$ attains the minimum.
\end{enumerate}
Suppose the orbit $G_{\phi}\cdot\mu$ is not closed. Then there exists a unique closed orbit $G_{\phi}\cdot\mu_0 \subset \overline{G_{\phi}\cdot\mu}$ such that the algebra $\ngo_{0}=(\RR^n,\mu_0)$ is an Einstein nilradical not isomorphic to $\ngo$ and there exists a symmetric matrix $A\in \ggo_{\phi}$, with integer eigenvalues, such that $\lim \limits_{t\rightarrow \infty} \exp(tA)\cdot\mu = \mu_0$

\end{theorem}

The inner product in $V$ (and so the norm used in Theorem \ref{nikoclosedorbit}) is the induced by the canonical inner product of $\RR^n$ (see \cite[Equation (3.3)]{LAURET3}).

\begin{definition}\label{nicebasis}
  Let $\{X_1,...,X_n\}$ be a basis for a nilpotent Lie algebra $\ngo$, with ${[X_i,X_j]}=\sum \limits_{k=1}^{n} c_{ij}^kX_k$. The basis $\{X_i\}$ is called \textit{nice}, if for all $i, \,j,$ $\# \{k : c_{ij}^k \neq 0 \} \leq 1$ and for all $i, \,k,$ $\#\{j : c_{ij}^k \neq 0\} \leq 1$.
\end{definition}

In dimension $6$, there is only one real nilpotent Lie algebra that does not admit a nice basis, it is denoted by $L_{6,11}$ in \cite{GRAAF1}. In fact, following Graaf's classification, $L_{6,11}$ is the only which is not written in a nice basis. The dimensions of the descending central series and derived series for $L_{6,11}$ are $(6,3,2,1,0)$ and $(6,3,0)$, respectively, and it is easy to see that if a $6$-dimensional nilpotent Lie algebra with the same dimensions of the descending central series and derived series that $L_{6,11}$ admit a nice basis, then any such algebra must be isomorphic to $L_{6,12}$, $L_{6,13}$, $L_{5,6}\oplus \RR x_6$ or $L_{5,7} \oplus \RR x_6$. Thus $L_{6,11}$ can not admit a nice basis.
In dimension $7$, any complex nilpotent Lie algebra of rank greater or equal than $3$ has a nice basis (by a straightforward verification of the list \cite{MAGNIN1}).  It is proved in \cite{NIKOLAYEVSKY1} that every filiform Lie algebra admitting an $\NN$-gradation have a nice basis. Although the nice basis condition looks \textquotedblleft{very exclusive}\textquotedblright, it is satisfied by many families of nilpotent Lie algebras.

Let $\{E_{ij}\}$ be the canonical basis of $\glg_n(\RR)$ and let $\alpha_{ij}^k=E_{kk}-E_{ii}-E_{jj}$ with $1\leq i < j \leq n$ and $1\leq k \leq n$ denote the weights of the representation of $\glg_n(\RR)$ on $V$ (via the action of $\Glr$ on $V$). Let us now state Nikolayevsky's nice basis criterium.

\begin{theorem}\cite[Theorem 3.]{NIKOLAYEVSKY2}\label{nikonicebasis}
 Let $\ngo=(\RR^n,\mu)$  be a nilpotent Lie algebra, with $\mu = \sum c_{ij}^k \mu_{ij}^k$, $\mu\neq 0$. Let $F$ be an ordered set of weights that are related with $\mu$ (i.e. $\alpha_{ij}^k \in F$ if and only if $c_{ij}^k\neq 0$), set $m=\# F$ and define the (Gram) matrix $U \in \textsc{M}(m,\RR)$ as
$$
U_{p,q} :=  \tr(F(p)F(q)).
$$
If the canonical basis $\{e_i\}_{i=1}^n$ is a nice basis then the following conditions are equivalent:

\begin{enumerate}[\hspace*{0cm} (i)]
\item $\ngo$ is an Einstein nilradical.
\item The vector of minimum norm in the convex hull of $F$ is in the interior of the hull.
\item The equation $Ux=[1]_{m}$ has at least one solution $x$ with positive coordinates.
\end{enumerate}

\end{theorem}

\section{Classification}\label{examples}

In this section, we study in detail some examples which cover all possible strategies we have used to obtain
the full classification of $7$-dimensional Einstein nilradicals. Each Lie algebra in \cite{MAGNIN1} has been worked out in some of these ways (see \cite{FERNANDEZ-CULMA2} or \cite{FERNANDEZ-CULMA1} for more information).

Undoubtedly, our principal tool is Nikolayevsky's nice basis criterium and Carles' classification. Following \cite{MAGNIN1}, the only algebras of $\rank \geq 1$ that are not written in a nice basis are $1.2(ii)$, $1.2(iv)$, $1.3(i_\lambda)$, $1.3(ii)$, $1.3(v)$, $1.11$, $1.12$, $1.13$, $1.14$, $1.15$, $1.16$, $1.17$, $1.18$, $1.21$, $2.2$, $2.11$, $2.24$, $2.25$, $2.26$, $2.27$, $2.37$. Some of these algebras may admit a nice basis or not. Recall that a $\rank$-zero nilpotent Lie algebra can not be an Einstein nilradical as it does not admit an $\NN$-gradation.

By using that any nilpotent Lie algebra of dimension less or equal than $6$ is an Einstein nilradical (\cite{LAURET2}, \cite{WILL1}), one obtains that any decomposable $7$-dimensional nilpotent Lie algebra is an Einstein nilradical.

\begin{theorem}
The classification of $7$-dimensional indecomposable Einstein nilradicals is given according to their rank in Tables $1$, $2$, $3$ and $4$.
\end{theorem}

The notation in the tables is as follows: \begin{quote} $\checkmark$:= Yes, -:= No, EN:= Einstein Nilradical, pre-Einstein derivation:=\textquotedblleft{we give a pre-Einstein derivation with respect to the canonical basis}\textquotedblright, Min:= Minimum of $\F|_{\Glr\cdot\mu}$ (computed by using formula \ref{betanorm} with $3$ decimals), $\dim$ DCS:= Dimension of descending central series and $\dim \, \Der$:= Dimension of the algebra of derivations. \end{quote} In $\dim$ DCS we omit the first term, which is always 7. So for example the algebra $\ngo = (\RR^n,\mu)$ with $\mu= 1.3(i_0)$ is not an Einstein nilradical and hence minimum of $\F|_{\Glr\cdot\mu}$ does not exist, the dimension of its algebra of derivations is $13$ and dimension of descending central series is $(7,4,2,1,0)$ that correspond to $\ngo = \ngo_{0} \geq \ngo_{1}=[\ngo,\ngo_{0}] \geq \ngo_{2}=[\ngo,\ngo_{1}] \geq \ngo_{3}=[\ngo,\ngo_{2}] \geq \ngo_{4}=[\ngo,\ngo_{3}] = 0$.

\begin{example}\label{example1}\textbf{Exhibiting a nilsoliton inner product} \newline
In this example we show how to exhibit a nilsoliton inner product. We consider $\mathfrak{g}_{1,17}$
$$
\mu:= \left\{\begin{array}{l}
[e_1, e_2] = e_3, [e_1, e_3] = e_4, [e_1, e_4] = e_6, [e_1, e_6] = e_7, {[e_2, e_3]} = e_5,\\
{[e_2, e_5]} = e_6, [e_2, e_6] = e_7, [e_3, e_4] = -e_7, {[e_3, e_5]} = e_7\\
\end{array}\right.
$$
This algebra has $\rank$ $1$ and its maximal torus of derivations is generated by $\diag (1,1,2,3,3,4,5)$. It follows from corollary \ref{beta} that if $\ggo_{1.17}$ is an Einstein nilradical, then there must be a Lie algebra law $\widetilde{\mu}$ in $\mathrm{GL}_7(\RR) \cdot \mu$ with $\diag (1,1,2,3,3,4,5)$ as Einstein derivation and moment map equal to
\begin{equation}\label{ee_1}
\mm(\widetilde{\mu})=\diag\left(-\frac{23}{47},-\frac{23}{47},-\frac{27}{94},-\frac{4}{47},-\frac{4}{47},\frac{11}{94},\frac{15}{47} \right)
\end{equation}

However, any algebra law admitting $\diag (1,1,2,3,3,4,5)$ as a derivation is of the form
$$
\mu(a_1,...,a_{13}):= \left\{\begin{array}{l}
[e_1,e_2]=a_1 e_3, [e_1,e_3]=a_2 e_4+a_3 e_5, \\
{[e_1,e_4]}=a_4 e_6, [e_1,e_5]=a_5 e_6, [e_1,e_6]=a_6 e_7, \\
{[e_2,e_3]}=a_7 e_4 + a_8 e_5, [e_2,e_4]=a_9 e_6, {[e_2,e_5]}=a_{10} e_6, \\
{[e_2,e_6]}=a_{11} e_7, [e_3,e_4]=a_{12} e_7, {[e_3,e_5]}=a_{13} e_7.
\end{array}\right.
$$
If $J$ represents the Jacobi condition, then Einstein nilradicals of eigenvalue type $(1<2<3<4<5 ; 2,1,2,1,1)$ are characterized by  $J(\mu(a_1,...,a_{13}))=0$ and  $\mm(\mu(a_1,...,a_{13}))$ as in (\ref{ee_1}), or equivalently, by the following polynomial equations system:

$$\begin{array}{ll}
J(\mu(a_1,...,a_{13})) &\left\{\begin{array}{l} -a_{10}a_{6}+a_{5}a_{11}+a_{1}a_{13}=0,\\ -a_{9}a_{6}+a_{4}a_{11}+a_{1}a_{12}=0,\\ -a_{8}a_{5}+a_{3}a_{10}-a_{7}a_{4}+a_{2}a_{9}=0,\end{array}\right.\\
\mm(\mu(a_1,...,a_{13})) & \left\{\begin{array}{l}
-2a_{1}^2-2a_{2}^2-2a_{3}^2-2a_{4}^2-2a_{5}^2-2a_{6}^2=-\frac{23}{47}, \\
-2 a_{2}a_{7}-2a_{3}a_{8}-2a_{4}a_{9}-2a_{5}a_{10}-2a_{6}a_{11} =0, \\
-2a_{1}^2-2a_{7}^2-2a_{8}^2-2a_{9}^2-2a_{10}^2-2a_{11}^2=-\frac{23}{47},\\
2a_{1}^2-2a_{2}^2-2a_{3}^2-2a_{7}^2-2a_{8}^2-2a_{12}^2-2a_{13}^2=-\frac{27}{94}, \\
2a_{2}^2-2a_{4}^2+2a_{7}^2-2a_{9}^2-2a_{12}^2=-\frac{4}{47},\\
2a_{2}a_{3}-2a_{4}a_{5}+2a_{7}a_{8}-2a_{9}a_{10}-2a_{12}a_{13}=0,\\
2a_{3}^2-2a_{5}^2+2a_{8}^2-2a_{10}^2-2a_{13}^2=-\frac{4}{47}, \\
2a_{4}^2+2a_{5}^2-2a_{6}^2+2a_{9}^2+2a_{10}^2-2a_{11}^2=\frac{11}{94},\\
2a_{6}^2+2a_{11}^2+2a_{12}^2+2a_{13}^2=\frac{15}{47}.
\end{array}\right.
\end{array}$$

Since an eigenvalue type may have many non isomorphic Einstein nilradicals, we must find a solution such that its nilpotent Lie algebra law is isomorphic to $\ggo_{1.17}$.  A priori, we have found a pre-Einstein derivation to all $7$-dimensional indecomposable algebras in \cite{MAGNIN1} and $\ggo_{1.17}$ is the only indecomposable nilpotent Lie algebra that has a pre-Einstein derivation of eigenvalue type $(1<2<3<4<5 ; 2,1,2,1,1)$, thus we only have to verify the indecomposability property between solutions of such equations system to show that $\ggo_{1.17}$ is an Einstein nilradical. By using Gr$\ddot{\mbox{o}}$bner basis to solve the polynomial equations system we find some solutions given by:
\begin{equation}\label{exsol}
\begin{array}{l}
a_1=\pm \frac{\sqrt {611}}{94},
a_2=0,
a_3=\pm \frac{\sqrt {235}}{47},
a_4= 0,
a_5=\pm \frac{\sqrt {611}}{94},
a_6=0,\\
a_7=\pm \frac{\sqrt {235}}{94},
a_8=0,
a_9=\pm \frac{\sqrt {611}}{94},
a_{10}=0,
a_{11}=\pm \frac {\sqrt {705}}{94},\\
a_{12}=0,
a_{13}=-\frac{a_5 a_{11}}{a_1}.
\end{array}
\end{equation}
By fixing a solution we have an Einstein nilradical $(\RR^n, \widetilde{\mu})$ given by
$$
\widetilde{\mu}:=\left\{\begin{array}{l}
[e_1, e_2] = \frac{\sqrt{611}}{94}e_3, [e_1,e_3] = \frac{\sqrt{235}}{47}e_5, [e_1, e_5] = \frac{\sqrt{611}}{94}e_6,  [e_2, e_3] = \frac{\sqrt{235}}{94}e_4, \\
{[e_2, e_4]} = \frac{\sqrt{611}}{94}e_6,  [e_2, e_6] = \frac{\sqrt{705}}{94}e_7, [e_3, e_5] = -\frac{\sqrt{705}}{94}e_7
\end{array}\right.
$$
As $(\RR^n, \widetilde{\mu})$ is indecomposable, by the analysis above, this must be isomorphic to $(\RR^n,{\mu})$. To find an isomorphism, since the map $\diag (1,1,2,3,3,4,5)$  with respect to the basis $\{e_i\}$ is a derivation of both algebras, we can assume that an isomorphism is given by a matrix in $\mathrm{GL}_7(\RR)$ that commutes with $\diag (1,1,2,3,3,4,5)$,

$$g=\diag \left( \left(\begin{array}{cc} b_{1,1} & b_{1,2} \\ b_{2,1} & b_{2,2} \end{array}\right) , b_{3,3} , \left(\begin{array}{cc} b_{4,4} & b_{4,5} \\ b_{5,4} & b_{5,5} \end{array}\right) , b_{6,6}, b_{7,7} \right) $$

By solving the equation $g\cdot \mu = \widetilde{\mu}$ we get
$$
g=\left( \begin {array}{ccccccc}
1&-1&0&0&0&0&0\\
\sqrt {2}&\sqrt {2}&0&0&0&0&0\\
0&0&\frac{\sqrt {1222}}{47}&0&0&0&0\\
0&0&0&\frac{\sqrt {65}}{47} &\frac{\sqrt {65}}{47}&0&0\\
0&0&0&\frac{\sqrt {130}}{47} & -\frac{\sqrt {130}}{47}&0&0\\
0&0&0&0&0&{\frac {13\sqrt {470}}{2209}}&0\\
0&0&0&0&0&0&{\frac {65\sqrt{3}}{2209}}
\end {array} \right),
$$
and thus $\ggo_{1.17}$ is an Einstein nilradical, as was to be shown.
\end{example}
\begin{center}
\begin{tabular}{|l|c|c|c|c|c|}
\hline
\multicolumn{6}{     |c|      }  {\centering Rank one }  \tabularnewline
\hline
\multicolumn{1}{ |m{1.5cm }|  }  {\centering $\mathfrak{n}$ } &
\multicolumn{1}{  m{0.5cm }|  }  {\centering EN } &
\multicolumn{1}{  m{4.5cm }|  }  {\centering pre-Einstein \\ derivation} &
\multicolumn{1}{  m{1cm}|     }  {\centering Min } &
\multicolumn{1}{  m{0.75cm }| }  {\centering $\dim$ $\Der$} &
\multicolumn{1}{  m{2cm }|    }  {\centering $\dim$ \\ DCS} \tabularnewline
\hline
\hline
$1.01(i)$&-&$(0, 1, 0, 1, 1, 1, 1)$ & - &$11$&$(4, 3, 2, 1)$ \\
\hline
$1.01(ii)$&-&$(0, 1, 0, 1, 1, 1, 1)$& - &$12$&$(4, 3, 2, 1)$ \\
\hline
$1.02$&-&$\frac{1}{2}(0, 1, 1, 1, 2, 2, 3)$& - &$11$&$(5, 4, 2, 1)$ \\
\hline
$1.03$&-&$\frac{2}{3}(0, 1, 1, 1, 1, 2, 2)$& - &$11$&$(5, 4, 2, 1)$\\
\hline
$1.1(i_\lambda)$  & \multirow{2}{*}{$\checkmark$} & \multirow{2}{*}{$\frac{1}{5}(1, 2, 3, 4, 5, 6, 7)$} & \multirow{2}{*}{$0.714$} & \multirow{2}{*}{$10$}& \multirow{2}{*}{$(5, 4, 3, 2, 1)$} \\
$\lambda\neq 0,1$& & & & & \\
\hline
$1.1(i_\lambda)$ &\multirow{2}{*}{-}&\multirow{2}{*}{$\frac{1}{5}(1, 2, 3, 4, 5, 6, 7)$}& \multirow{2}{*}{-}&\multirow{2}{*}{$10$}&\multirow{2}{*}{$(5, 4, 3, 2, 1)$}\\
$\lambda = 0$    & & & & & \\
\hline
$1.1(i_\lambda)$ &\multirow{2}{*}{-}&\multirow{2}{*}{$\frac{1}{5}(1, 2, 3, 4, 5, 6, 7)$}&\multirow{2}{*}{ - }&\multirow{2}{*}{$11$}&\multirow{2}{*}{$(5, 4, 3, 2, 1)$ }\\
$\lambda = 1$ &&&&& \\
\hline
$1.1(ii)$&-&$\frac{1}{5}(1, 2, 3, 4, 5, 6, 7)$& - & $11$& $(5, 4, 3, 2, 1)$\\
\hline
$1.1(iii)$&$\checkmark$&$\frac{1}{5}(1, 2, 3, 4, 5, 6, 7)$&$0.714$&$10$&$(5, 4, 3, 2)$ \\
\hline
$1.1(iv)$&-&$\frac{1}{5}(1, 2, 3, 4, 5, 6, 7)$& - &$11$&$(5, 4, 2, 1)$ \\
\hline
$1.1(v) $&-&$\frac{1}{5}(1, 2, 3, 4, 5, 6, 7)$& - & $10$ &$(4, 3, 2, 1)$ \\
\hline
\end{tabular}
\end{center}
\begin{center}
\begin{tabular}{|l|c|c|c|c|c|}
\hline
\multicolumn{6}{     |c|      }  {\centering Rank one }  \tabularnewline
\hline
\multicolumn{1}{ |m{1.5cm }|  }  {\centering $\mathfrak{n}$ } &
\multicolumn{1}{  m{0.5cm }|  }  {\centering EN } &
\multicolumn{1}{  m{4.5cm }|  }  {\centering pre-Einstein \\ derivation} &
\multicolumn{1}{  m{1cm}|     }  {\centering Min } &
\multicolumn{1}{  m{0.75cm }| }  {\centering $\dim$ $\Der$} &
\multicolumn{1}{  m{2cm }|    }  {\centering $\dim$ \\ DCS} \tabularnewline
\hline
\hline
$1.1(vi)$&-&$\frac{1}{5}(1, 2, 3, 4, 5, 6, 7)$ & - &$11$& $(4, 3, 2, 1)$ \\
\hline
$1.2(i_\lambda)$&\multirow{2}{*}{ $\checkmark$ } &\multirow{2}{*}{$\frac{4}{11}(1, 1, 2, 2, 3, 3, 4)$}&\multirow{2}{*}{$0.846$}&\multirow{2}{*}{$12$}&\multirow{2}{*}{$(4, 3, 1)$} \\
$\lambda \neq 0 , 1$&&&&& \\
\hline
$1.2(i_\lambda)$ &\multirow{2}{*}{$\checkmark$}&\multirow{2}{*}{$\frac{4}{11}(1, 1, 2, 2, 3, 3, 4)$}&\multirow{2}{*}{$0.846$}&\multirow{2}{*}{$12$}&\multirow{2}{*}{$(4, 3, 1)$} \\
$\lambda = 0$    &&&&& \\
\hline
$1.2(i_\lambda)$&\multirow{2}{*}{ $\checkmark$ } &\multirow{2}{*}{$\frac{4}{11}(1, 1, 2, 2, 3, 3, 4)$}&\multirow{2}{*}{$0.846$}&\multirow{2}{*}{$12$}&\multirow{2}{*}{ $(4, 3, 1)$}\\
$\lambda = 1$   &&&&& \\
\hline
$1.2(ii)$&-&$\frac{4}{11}(1, 1, 2, 2, 3, 3, 4)$& - & $12$ & $(4, 3, 1)$ \\
\hline
$1.2(iii)$&-&$\frac{4}{11}(1, 1, 2, 2, 3, 3, 4)$& - &$12$& $(4, 3, 1)$\\
\hline
$1.2(iv)$&-&$\frac{4}{11}(1, 1, 2, 2, 3, 3, 4)$& - & $12 $ &  $(4, 2, 1)$ \\
\hline
$1.3(i_\lambda)$&\multirow{2}{*}{$\checkmark$}&\multirow{2}{*}{$\frac{5}{17}(1, 2, 2, 3, 3, 4, 5)$}&\multirow{2}{*}{$0.895$}&\multirow{2}{*}{$13$}&\multirow{2}{*}{ $(4, 2, 1)$}\\
$\lambda \neq 0$&&&&& \\
\hline
$1.3(i_\lambda)$&\multirow{2}{*}{-}&\multirow{2}{*}{$\frac{5}{17}(1, 2, 2, 3, 3, 4, 5)$}&\multirow{2}{*}{-}
&\multirow{2}{*}{$13$}&\multirow{2}{*}{ $(4, 2, 1)$}\\
$\lambda = 0$&&&&& \\
\hline
$1.3(ii)$&-&$\frac{5}{17}(1,2,2,3,3,4,5)$ & - & $14$ & $(4, 2, 1)$\\
\hline
$1.3(iii)$&$\checkmark$&$\frac{5}{17}(1, 2, 2, 3, 3, 4, 5)$&$0.895$&$13$& $(4, 2, 1)$\\
\hline
$1.3(iv)$&-&$\frac{5}{17}(1, 2, 2, 3, 3, 4, 5 )$& - & $13$& $(4, 2)$ \\
\hline
$1.3(v)$ &-&$\frac{5}{17}( 1, 2, 2, 3, 3, 4, 5)$& - & $13$& $(3, 2, 1)$\\
\hline
$1.4$    &$\checkmark$&$\frac{17}{100}(1, 3, 4, 5, 6, 7, 8)$&$0.820$& $12$& $(5, 4, 3, 2, 1)$\\
\hline
$1.5$    &$\checkmark$&$\frac{5}{31}(1, 3, 4, 5, 6, 7, 9) $&$0.738$&$ 11$&$(5, 4, 3, 2)$ \\
\hline
$1.6$    &$\checkmark$&$\frac{5}{34}(1, 4, 5, 6, 7, 8, 9)$&$0.895$&$12$& $(5, 4, 3, 2, 1)$\\
\hline
$1.7$    &$\checkmark$&$\frac{5}{29}(2, 3, 4, 5, 6, 7, 8)$&$1.04$&$15$&$(4, 2)$ \\
\hline
$1.8$    &-&$\frac{20}{139}(2, 4, 3, 6, 7, 8, 10)$& - &$11$& $(4, 2, 1)$\\
\hline
$1.9$    &-&$\frac{10}{67}(2, 3, 6, 5, 7, 8, 9)$  &  -   &$14$& $(4, 3, 1)$\\
\hline
$1.10$   &$\checkmark$&$\frac{45}{353}(2, 3, 5, 7, 8, 9, 11)$&$0.792$&$11$& $(5, 4, 2, 1)$\\
\hline
$1.11$   &$\checkmark$&$\frac{6}{25}(1,2,3,3,4,5,6)$    &$0.806$&$11$&  $(4, 3, 2, 1)$\\
\hline
$1.12$   &$\checkmark$&$\frac{25}{107}(1,2,4,3,4,5,6)$  &$0.863$&$12$&  $(4, 3, 2, 1)$\\
\hline
$1.13$   &$\checkmark$&$\frac{13}{58}(1,2,3,4,5,5,6)$   &$0.853$&$12$&  $(5, 4, 2, 1)$\\
\hline
$1.14$   &$\checkmark$&$\frac{ 9}{43}(1,2,3,4,5,5,7)$    &$0.741$&$11$&  $(5, 4, 2, 1)$ \\
\hline
\end{tabular}
\end{center}

\begin{center}
\begin{tabular}{|l|c|c|c|c|c|}
\hline
\multicolumn{6}{     |c|      }  {\centering Rank one }  \tabularnewline
\hline
\multicolumn{1}{ |m{1.5cm }|  }  {\centering $\mathfrak{n}$ } &
\multicolumn{1}{  m{0.5cm }|  }  {\centering EN } &
\multicolumn{1}{  m{4.5cm }|  }  {\centering pre-Einstein \\ derivation} &
\multicolumn{1}{  m{1cm}|     }  {\centering Min } &
\multicolumn{1}{  m{0.75cm }| }  {\centering $\dim$ $\Der$} &
\multicolumn{1}{  m{2cm }|    }  {\centering $\dim$ \\ DCS} \tabularnewline
\hline
\hline
$1.15$   &$\checkmark$&$\frac{15}{76}(1,3,4,4,5,6,7)$   &$0.927$&$13$&  $(4, 3, 2, 1)$ \\
\hline
$1.16$   &$\checkmark$&$\frac{11}{40}(1,2,3,3,4,4,5)$   &$1.05 $&$15$&  $(4, 2, 1)$\\
\hline
$1.17$   &$\checkmark$&$\frac{19}{65}(1,1,2,3,3,4,5)$   &$0.692$&$11$&  $(5, 4, 2, 1)$\\
\hline
$1.18$   &$\checkmark$&$\frac{23}{89}(1,2,3,3,4,5,5)$   &$0.947$&$13$&  $(4, 3, 1)$\\
\hline
$1.19$   &$\checkmark$&$\frac{13}{29}(1,1,1,2,2,3,3)$   &$0.853$&$11$&  $(4, 2)$\\
\hline
$1.20$   &-&$\frac{8}{47}(1,2,3,5,6,7,8)$    &   -   &$11$&  $(4, 3, 2, 1)$\\
\hline
$1.21$   &-&$\frac{25}{113}(1,2,3,3,4,5,7)$  &   -   &$11$&  $(4, 3, 2, 1)$\\
\hline
\multicolumn{6}{c}  {\footnotesize{\textsc{Table 1.} Classification of $7$-dimensional indecomposable}}  \tabularnewline
\multicolumn{6}{c}  {\footnotesize{Einstein nilradicals. Rank one case.}}  \tabularnewline
\end{tabular}
\end{center}

\begin{example}\label{example0} \textbf{Showing that the $G_{\phi}$-orbit is closed} \newline
In this example we consider exclusively $\ggo_{1.3(i_{\lambda})}$. This one-parameter family is the only one that cannot be covered by the others examples.
$$
\mu:= \left\{ \begin{array}{l}
[e_1,e_2]=e_4,[e_1,e_3]=e_5,[e_1,e_4]=e_6,[e_1,e_6]=e_7,[e_2,e_3]=e_6,\\
{[e_2,e_4]}=\lambda e_7,[e_2,e_5]=e_7,[e_3,e_5]=e_7 \end{array} \right.
$$
For any $\lambda \neq 0$, $\ggo_{1.3(i_{\lambda})}$ is an Einstein nilradical. We prove this by contradiction; assume that $\ggo_{1.3(i_{\lambda})}$ is not an Einstein nilradical. The derivation $\phi$ given by the diagonal matrix $\frac{5 }{17 }\diag( 1, 2, 2, 3, 3, 4, 5 )$ with respect to the basis $\{e_i\}$ is pre-Einstein. It follows from Theorem \ref{nikoclosedorbit} that the orbit $G_{\phi}\cdot\mu$ is not closed and so there exists $Y\in \ggo_{\phi}$ , $Y$ is a symmetric matrix, such that $\mu$ degenerates by the action of the one-parameter subgroup $\exp(tY)$ as $t\rightarrow \infty$. As $Y\in \ggo_{\phi}$ and $Y$ is symmetric, then there exists $X=\diag(a_1,...,a_7) \in \ggo_{\phi}$ and $A(\alpha)$, $B(\beta)$ in $\SO_2(\RR)$ such that
$$
Y=\diag(1,A(-\alpha),B(-\beta),1,1) X \diag(1,A(\alpha),B(\beta),1,1).
$$
As the action is continuous, it follows that $\mu$ also degenerates by the action of
$$
g_t:=\exp(tX) \diag(1,A(\alpha),B(\beta),1,1)
$$
as $t\rightarrow \infty$. The contradiction will be found in this last fact. The action of $g_t$ in $\mu$ is
$$
g_t\cdot\mu =\left\{\begin{array}{l}
{[e_1,e_2]}={\euler^{-t \left( a_1+a_2-a_4 \right) }}\cos \left( \alpha-\beta \right) e_4-{\euler^{-t \left( a_1+a_2-a_5 \right) }}\sin \left( \alpha-\beta \right) e_5,\\
{[e_1,e_3]}={\euler^{-t \left( a_1+a_3-a_4 \right) }}\sin \left( \alpha-\beta \right) e_4+{\euler^{-t \left( a_1+a_3-a_5 \right) }}\cos \left( \alpha-\beta \right) e_5,\\
{[e_1,e_4]}={\euler^{-t \left( a_1+a_4-a_6 \right) }}\cos \left( \beta \right) e_6,[e_1,e_5]={\euler^{-t\left( a_1+a_5-a_6 \right) }}\sin \left( \beta \right) e_6,\\
{[e_1,e_6]}={\euler^{-t \left( a_1+a_6-a_7 \right) }}e_7,[e_2,e_3]={\euler^{-t \left( a_2+a_3-a_6 \right) }}e_6,\\
{[e_2,e_4]}={\euler^{-t \left( a_2+a_4-a_7 \right) }} f_{2,4}(\alpha,\beta)e_7,\\
{[e_2,e_5]}={\euler^{-t \left( a_2+a_5-a_7 \right) }} f_{2,5}(\alpha,\beta)e_7,\\
{[e_3,e_4]}={\euler^{-t \left( a_3+a_4-a_7 \right) }} f_{3,4}(\alpha,\beta)e_7,\\
{[e_3,e_5]}={\euler^{-t \left( a_3+a_5-a_7 \right) }} f_{3,5}(\alpha,\beta)e_7.
\end{array}
\right.
$$
with
$$
\begin{array}{l}
f_{2,4}(\alpha, \beta)=(\cos \left( \alpha \right) \cos \left( \beta \right) \lambda + \sin \left( \alpha \right) \sin \left( \beta \right) -\cos \left( \alpha \right) \sin \left( \beta \right) ),\\
f_{2,5}(\alpha, \beta)=(\cos \left( \alpha \right) \sin \left( \beta \right) \lambda - \sin \left( \alpha \right) \cos \left( \beta \right) +\cos \left( \alpha \right) \cos \left( \beta \right) ),\\
f_{3,4}(\alpha, \beta)=(\sin \left( \alpha \right) \cos \left( \beta \right) \lambda - \cos \left( \alpha \right) \sin \left( \beta \right) -\sin \left( \alpha \right) \sin \left( \beta \right) ),\\
f_{3,5}(\alpha, \beta)=(\sin \left( \alpha \right) \sin \left( \beta \right) \lambda + \cos \left( \alpha \right) \cos \left( \beta \right) +\sin \left( \alpha \right) \cos \left( \beta \right) ).
\end{array}
$$
According to values of $\alpha$ and $\beta$ some terms are zero in the Lie algebra law $g_t\cdot \mu$ and as the degeneration is determined by non-zero terms, our attention is in the exponent of the exponential factor of such terms; when $t>0$, such exponent must be non negative.

It is easy to see that pairs of functions $\{f_{2,4},f_{2,5}\}$, $\{f_{2,4},f_{3,4}\}$, $\{f_{2,5},f_{3,5}\}$ and $\{f_{3,4},f_{3,5}\}$ do not vanish simultaneously (as $\sin$ and $\cos$). We have the following cases depending on which terms are non zero.

\

\noindent I) $\cos(\beta)$ and $\sin(\beta)$ are non zero \newline
1. $\cos(\alpha-\beta),\,f_{2,4},\,f_{3,4}$ are non zero \newline

If this is the case then it must be that $a_1,...,a_7$ satisfy: $a_1+a_2+a_3+a_4+a_5+a_6+a_7=0$, $a_1+2a_2+2a_3+3a_4+3a_5+4a_6+5a_7=0$ since $X\in \ggo_{\phi}$ and $a_1+a_6-a_7\geq 0$, $a_2+a_3-a_6 \geq 0$, $a_1+a_4-a_6\geq 0$, $a_1+a_5-a_6\geq 0$, $a_1+a_2-a_4 \geq 0$, $a_1+a_3-a_5 \geq 0$, $ a_2+a_4-a_7 \geq 0$ and $a_3+a_4-a_7 \geq 0$.

Instead of solving the inequalities system, we can do the next trick: we introduce a new variable $c_i$ for each inequality $q_i$ and we have $q_i - c_i^2=0$ and so we must solve the polynomial equations system:
$$
\left\{\begin{array}{l}
a_1+a_2+a_3+a_4+a_5+a_6+a_7=0,\\
a_1+2a_2+2a_3+3a_4+3a_5+4a_6+5a_7=0,\\
a_1+a_6-a_7 - c_1^2=0,\,a_2+a_3-a_6 - c_2^2=0,\\
a_1+a_4-a_6 - c_3^2=0,\,a_1+a_5-a_6 - c_4^2=0,\\
a_1+a_2-a_4 - c_5^2=0,\,a_1+a_3-a_5 - c_6^2=0,\\
a_2+a_4-a_7 - c_7^2=0,\,a_3+a_4-a_7 - c_8^2=0.\\
\end{array}\right.
$$
We set that $c_1^2+4c_4^2+5c_5^2+2c_6^2+2c_7^2+5c_8^2=0$, $c_2^2-6c_5^2-6c_8^2-4c_4^2-2c_6^2-2c_7^2=0$ and $c_3^2+c_5^2+c_8^2-c_4^2-c_6^2-c_7^2=0$. By the first equality  $c_1$, $c_4$, $c_5$, $c_6$, $c_7$ and $c_8$ are zero and so $c_2$ and $c_3$ are zero too. The degeneration is therefore trivial and so in this case we get a contradiction.
Remaining cases are similar and we can see them in \cite{FERNANDEZ-CULMA2} or \cite{FERNANDEZ-CULMA1}. So, no matter which case is, the degeneration is trivial. Hence $\ggo_{1.3(i_{\lambda})}$, with $\lambda \neq 0$, must be an Einstein Nilradical.
\end{example}
\begin{center}
\begin{tabular}{|l|c|c|c|c|c|}
\hline
\multicolumn{6}{     |c|      }  {\centering Rank two }  \tabularnewline
\hline
\multicolumn{1}{ |m{1.5cm }|  }  {\centering $\mathfrak{n}$ } &
\multicolumn{1}{  m{0.5cm }|  }  {\centering EN } &
\multicolumn{1}{  m{4.5cm }|  }  {\centering pre-Einstein \\ derivation} &
\multicolumn{1}{  m{1cm}|     }  {\centering Min } &
\multicolumn{1}{  m{0.75cm }| }  {\centering $\dim$ $\Der$} &
\multicolumn{1}{  m{2cm }|    }  {\centering $\dim$ \\ DCS} \tabularnewline
\hline
\hline
$2.1(i_\lambda)$&\multirow{2}{*}{$\checkmark$}&\multirow{2}{*}{$\frac{2}{19}(3,5,6,8,9,11,14)$}&\multirow{2}{*}{$0.905$}&\multirow{2}{*}{$14$}&\multirow{2}{*}{$(4, 2, 1)$}\\
$\lambda \neq 0,1$&&&&&\\
\hline
$2.1(i_\lambda)$&\multirow{2}{*}{-}&\multirow{2}{*}{$\frac{2}{19}(3,5,6,8,9,11,14)$}&\multirow{2}{*}{-}&\multirow{2}{*}{$14$}&\multirow{2}{*}{$(4, 2, 1)$}\\
$\lambda=0 $&&&&&\\
\hline
\end{tabular}
\end{center}

\begin{center}
\begin{tabular}{|l|c|c|c|c|c|}
\hline
\multicolumn{6}{     |c|      }  {\centering Rank two }  \tabularnewline
\hline
\multicolumn{1}{ |m{1.5cm }|  }  {\centering $\mathfrak{n}$ } &
\multicolumn{1}{  m{0.5cm }|  }  {\centering EN } &
\multicolumn{1}{  m{4.5cm }|  }  {\centering pre-Einstein \\ derivation} &
\multicolumn{1}{  m{1cm}|     }  {\centering Min } &
\multicolumn{1}{  m{0.75cm }| }  {\centering $\dim$ $\Der$} &
\multicolumn{1}{  m{2cm }|    }  {\centering $\dim$ \\ DCS} \tabularnewline
\hline
\hline
$2.1(i_\lambda)$&\multirow{2}{*}{$\checkmark$}&\multirow{2}{*}{$\frac{2}{19}(3,5,6,8,9,11,14)$}&\multirow{2}{*}{$0.905$}&\multirow{2}{*}{$14$}&\multirow{2}{*}{$(4, 2, 1)$}\\
$\lambda=1$&&&&&\\
\hline
$2.1(ii)$       &$\checkmark$&$\frac{2}{19}(3,5,6,8,9,11,14)$&$0.905$&$14$&$(4, 2, 1)$\\
\hline
$2.1(iii)$      &$\checkmark$&$\frac{2}{19}(3,5,6,8,9,11,14)$&$0.905$&$14$&$(3, 2, 1)$\\
\hline
$2.1(iv)$       &-&$\frac{2}{19}(3,5,6,8,9,11,14)$&  -  &$14$&$(3, 1)$\\
\hline
$2.1(v)$        &-&$\frac{2}{19}(3,5,6,8,9,11,14)$&  -  &$14$&$(4, 2)$\\
\hline
$2.2$           &-&$\frac{1}{2}(1,1,1,2,2,2,3)   $&  -  &$15$&$(4, 1)$\\
\hline
$2.3$           &$\checkmark$&$\frac{2}{37}(1,16,17,18,19,20,21)$&$1.06$ &$13$&$(5,4,3,2,1)$\\
\hline
$2.4$           &$\checkmark$&$\frac{7}{52}(1,4,5,6,7,8,11)$     &$0.743$&$12$&$(5,4,3,2)$\\
\hline
$2.5$           &$\checkmark$&$\frac{1 }{5 }( 1, 2, 3, 4, 5, 6, 7)$&$0.714$&$12$&$(5, 4, 2, 1)$\\
\hline
$2.6 $          &$\checkmark$&$\frac{ 1}{ 52}(10, 23, 33, 43, 56, 53, 76 )$&$0.743$&$12$&$(5, 4, 2, 1)$\\
\hline
$2.7 $          &$\checkmark$&$\frac{1 }{18 }( 3, 10, 13, 16, 23, 19, 22 )$&$0.9$&$13$&$(5, 4, 2, 1)$\\
\hline
$2.8$           &$\checkmark$&$\frac{1 }{12 }(3, 5, 8, 11, 13, 14, 16 )   $&$0.857$&$13$&$(5, 4, 2)$\\
\hline
$2.9$           &$\checkmark$&$\frac{9 }{28 }(1, 1, 2, 3, 3, 4, 4 )$&$0.824$&$12$&$(5, 4, 2)$\\
\hline
$2.10$          &-&$\frac{ 1}{ 5}(1, 2, 6, 3, 4, 5, 7 )$  &  -  &$12$&$(4, 3, 2, 1)$\\
\hline
$2.11$          &$\checkmark$&$\frac{ 1}{ 36}(9,19,28,28,37,47,46)$  &$0.947$&$14$&$(4, 3, 1)$\\
\hline
$2.12$          &$\checkmark$&$\frac{1 }{9 }(3, 5, 5, 8, 8, 11, 13 )$&$0.9$&$14$&(4, 2)\\
\hline
$2.13$          &$\checkmark$&$\frac{1 }{60 }( 16, 21, 48, 37, 53, 69, 90)$&$0.698$&$12$&$(4, 3, 2, 1)$\\
\hline
$2.14$          &$\checkmark$&$\frac{1 }{5 }(1, 3, 2, 4, 5, 6, 7 )$&$0.714$&$12$&$(4, 3, 2, 1)$\\
\hline
$2.15$          &$\checkmark$&$\frac{ 1}{ 5}(1, 3, 3, 4, 5, 6, 7 )$&$0.833$&$ 13$&$(4, 3, 2, 1)$\\
\hline
$2.16$          &$\checkmark$&$\frac{1 }{ 27}(5, 17, 20, 22, 27, 32, 37 )$&$0.931$&$14$&$(4, 3, 2, 1)$\\
\hline
$2.17$          &$\checkmark$&$\frac{1 }{12 }(4, 5, 8, 9, 13, 14, 17 )$&$0.857$&$13$&$(4, 3, 1)$\\
\hline
$2.18$          &$\checkmark$&$\frac{1 }{68 }( 20, 31, 60, 51, 71, 82, 91)$&$0.971$&$15$&$(4, 3, 1)$\\
\hline
$2.19$          &-&$\frac{1 }{4 }(1, 2, 4, 3, 4, 5, 5 )        $&  -    &$15$&$(4, 3, 1)$\\
\hline
$2.20$          &$\checkmark$&$\frac{ 2}{37 }(5, 16, 10, 21, 15, 20, 25)$  &$1.06$ &$16$&$(4, 2, 1)$\\
\hline
$2.21 $         &$\checkmark$&$\frac{1 }{31 }(8, 19, 24, 27, 32, 35, 43 )$ &$1.07$ &$16$&$(4, 2, 1)$\\
\hline
$2.22 $         &$\checkmark$&$\frac{1 }{19 }( 5, 14, 10, 15, 24, 20, 25)$ &$0.950$&$14 $&$(4, 2, 1)$\\
\hline
$2.23 $         &-&$\frac{ 4}{11 }(1, 1, 2, 2, 3, 3, 4 )$ &  -   &$ 13$&$(3, 1)$\\
\hline
$2.24 $         &$\checkmark$&$\frac{1}{17}(5,9,10,14,19,19,24)$&$0.895$&$13$&$(4, 2, 1)$\\
\hline
$ 2.25$         &$\checkmark$&$\frac{5}{17}(1,2,2,3,3,4,5)     $&$0.895$&$14$&$(3, 2, 1)$\\
\hline
$2.26$          &$\checkmark$&$\frac{1}{17}(7,10,7,17,14,21,24)$&$0.895$&$13$&$(4, 2)$\\
\hline
$2.27 $         &$\checkmark$&$\frac{1}{15}(6,7,14,13,13,19,20)$&$1.15$&$17$&$(3, 2)$\\
\hline
$2.28 $         &$\checkmark$&$\frac{ 4}{37 }(4,5,6,8,9,10,14)$&$1.06$&$16$&$(3, 1)$\\
\hline
$2.29 $         &-&$\frac{1 }{41 }(15,22,30,29,37,52,59)$&  -   &$14$&$(3, 2)$\\
\hline
$2.30$          &$\checkmark$&$\frac{4 }{27 }(2, 4, 5, 5, 6, 8, 10)$&$0.931$&$15$&$(3, 2, 1)$\\
\hline
$2.31$          &$\checkmark$&$\frac{1 }{39 }(14, 15, 27, 29, 42, 43, 57 )$&$0.848$&$13$&$(4, 2, 1)$\\
\hline
$2.32$          &$\checkmark$&$\frac{2 }{27 }(3, 10, 8, 13, 11, 16, 19 )  $&$0.931$&$14$&$(4, 2, 1)$\\
\hline
$2.33$          &$\checkmark$&$\frac{1 }{33 }(10, 18, 15, 28, 33, 38, 48 )$&$0.805$&$12$&$(4, 2, 1)$\\
\hline
$2.34$          &$\checkmark$&$\frac{1 }{47 }( 22, 20, 21, 42, 43, 62, 64)$&$0.854$&$12$&$(4, 2)$\\
\hline
\end{tabular}
\end{center}
\begin{center}
\begin{tabular}{|l|c|c|c|c|c|}
\hline
\multicolumn{6}{     |c|      }  {\centering Rank two }  \tabularnewline
\hline
\multicolumn{1}{ |m{1.5cm }|  }  {\centering $\mathfrak{n}$ } &
\multicolumn{1}{  m{0.5cm }|  }  {\centering EN } &
\multicolumn{1}{  m{4.5cm }|  }  {\centering pre-Einstein \\ derivation} &
\multicolumn{1}{  m{1cm}|     }  {\centering Min } &
\multicolumn{1}{  m{0.75cm }| }  {\centering $\dim$ $\Der$} &
\multicolumn{1}{  m{2cm }|    }  {\centering $\dim$ \\ DCS} \tabularnewline
\hline
\hline
$2.35$          &$\checkmark$&$\frac{1 }{13 }(5, 6, 7, 11, 12, 17, 18 )$   &$0.867$&$12$&$(4, 2)$\\
\hline
$2.36 $         &$\checkmark$&$\frac{ 1}{ 23}(18, 13, 10, 15, 28, 23, 33 )$&$1.10 $&$16$&$(3, 1)$\\
\hline
$2.37$          &$\checkmark$&$\frac{4 }{11 }(1, 1, 2, 2, 3, 3, 4)$&$0.846$ &$13$&$(4, 3, 1)$\\
\hline
$2.38$          &$\checkmark$&$\frac{ 7}{16 }(1, 1, 2, 2, 2, 3, 3 )$&$1.14$&$16$&$(3, 2)$\\
\hline
$2.39$          &$\checkmark$&$\frac{ 1}{16 }(5, 11, 10, 16, 15, 21, 20 )$&$1.14$&$17$&$(4, 2)$ \\
\hline
$2.40$          &$\checkmark$&$\frac{1 }{23 }(9, 10, 19, 18, 28, 29, 27)$&$1.10$&$16 $&$(4, 2)$\\
\hline
$2.41$          &$\checkmark$&$\frac{ 1}{7 }(2, 3, 5, 6, 7, 8, 10 )$&$0.875$&$13 $&$(4, 3, 1)$\\
\hline
$2.42$          &-&$\frac{1 }{41 }(11, 22, 30, 33, 41, 52, 55 )$&  -  &$14$&$(4, 2)$\\
\hline
$2.43$          &$\checkmark$&$\frac{1 }{ 37}(11, 29, 20, 40, 31, 42, 51 )$&$1.06$&$16$&$(4, 2)$\\
\hline
$2.44$          &$\checkmark$&$\frac{1 }{37 }( 15, 19, 23, 34, 38, 42, 53)$&$1.06$&$16$&$(4, 1)$\\
\hline
$2.45$          &$\checkmark$&$\frac{1 }{14}(6, 7, 11, 12, 13, 19, 18 )$&$1.17$&$17$&$(3, 1)$\\
\hline
\multicolumn{6}{c}  {\footnotesize{\textsc{Table 2.} Classification of $7$-dimensional indecomposable}}  \tabularnewline
\multicolumn{6}{c}  {\footnotesize{Einstein nilradicals. Rank two case.}}  \tabularnewline
\end{tabular}
\end{center}

\begin{example} \textbf{Degeneration by action of the group $G_{\phi}$} \newline
In this example, we show how to find a non-trivial degeneration by acting with a one-parameter diagonal  subgroup. We consider $\mathfrak{g}_{2.2}$
$$
\mu:= \left\{\begin{array}{l}
[e_1, e_2] = e_5,   [e_1, e_3] = e_6,    [e_1, e_4] = 2e_7,  [e_2, e_3] = e_4, {[e_2, e_6]} = e_7,  \\
{[e_3, e_5]} = -e_7,  [e_3, e_6] = e_7
\end{array} \right.
$$
$\ggo_{2.2}$ is a nilpotent Lie algebra of rank $2$ with a maximal torus of derivations generated by $D_1= \diag(1,0,0,0,1,1,1)$ and $D_2=\diag(0,1,1,2,1,1,2)$ (with respect to the basis $\{e_i\}$). If $\phi=aD_1+bD_2$ is a pre-Einstein derivation, then $a,b$ are found by solving the linear equations system
$$
\left\{\begin{array}{l}
\tr(\phi D_1) = \tr(D_1),\\
\tr(\phi D_2) = \tr(D_2).
\end{array}\right.
$$
We get that $\phi=\frac{1}{2}D_1+\frac{1}{2}D_2 = \frac{1}{2}(1,1,1,2,2,2,3)$, and so if $\ggo_{2.2}$ is an Einstein nilradical, then it has eigenvalue type $(1<2<3 ; 3,3,1)$. The method used in Example \ref{example1} applied to this algebra is very cumbersome because the polynomial equations system attached to the type  $(1<2<3 ; 3,3,1)$ has infinitely many solutions whose algebras are pairwise non isomorphic; the one-parameter family $3.1(\lambda)$ with $\lambda \neq 0 \neq 1$ is Einstein nilradical of type $(1<2<3 ; 3,3,1)$. By analyzing solutions of such system, we see that there is no any solution that corresponds to $\ggo_{2.2}$. Therefore we use another way to proof that $\ggo_{2.2}$ is not an Einstein Nilradical. By Theorem \ref{nikoclosedorbit}, we try to find a non-trivial degeneration of $\ggo_{2.2}$ by the action of a $1$-parameter diagonal subgroup of $G_{\phi}$. Let $X\in \ggo_{\phi}$ be a diagonal matrix, $X=\diag(a_1,...,a_7)$. As $\tr (X \phi) =0$ and $\tr(X)=0$ then $a_4=-a_5-a_6-2a_7$ and $a_1=a_7-a_2-a_3$. The action of $g_t=\exp(tX)$ on $\mu$ is
\begin{equation*}
\mu_t:= \left\{\begin{array}{l}
{[e_1, e_2]}=  \euler^{t  ( -a_7+ a_3 + a_5) }e_5,
{[e_1, e_3]}=  \euler^{t  ( -a_7+ a_2 + a_6 )}e_6, \\
{[e_1, e_4]}= 2\euler^{t  (  a_2+ a_3 + a_5 + a_6 + 2a_7 ) } e_7,
{[e_2, e_3]}=  \euler^{t  ( -a_2- a_3 - a_5 - a_6 - 2a_7 ) } e_4, \\
{[e_2, e_6]}=  \euler^{t  (  a_7- a_2 - a_6 ) } e_7,
{[e_3, e_5]}= -\euler^{t  (  a_7- a_3 - a_5 ) } e_7,\\
{[e_3, e_6]}=  \euler^{t  (  a_7- a_3 - a_6 ) } e_7
\end{array} \right.
\end{equation*}
To find a non-trivial degeneration as $t \rightarrow \infty$, we need that exponents be negative. By doing the same trick as in Example \ref{example0}, we get a polynomial equations system
$$
\left\{\begin{array}{l}
-a_7+ a_3 + a_5 = -b_1^2,\, -a_7+ a_2 + a_6 = -b_2^2 \\
 a_2+ a_3 + a_5 + a_6 + 2a_7 = -b_3^2 \\
-a_2- a_3 - a_5 - a_6 - 2a_7 = -b_4^2 \\
 a_7- a_2 - a_6 = -b_5^2,\, a_7- a_3 - a_5 = -b_6^2 \\
 a_7- a_3 - a_6 = -b_8^2 \\
\end{array} \right.
$$
whose solutions are given by
\begin{equation}
\begin{array}{l}
a_2 = \frac{1}{4}b_4^2 - a_5 - b_7^2 + \frac{3}{4} b_5^2 +\frac{3}{4}b_6^2, \,
a_3 = -a_5 + \frac{1}{4}b_4^2 - \frac{1}{4}b_5^2 + \frac{3}{4}b_6^2, \\
a_5 = a_5, \,
a_6 = a_5 + b_7^2 - b_6^2, \,
a_7 = \frac{1}{4}b_4^2 - \frac{1}{4}b_5^2 - \frac{1}{4}b_6^2, \,
b_1 = \pm \im b_6, \\
b_2 = \pm \im b_5, \,
b_3 = \pm \im b_4, \,
b_4 = b_4, \,
b_5 = b_5, \,
b_6 = b_6, \,
b_7 = b_7
\end{array}
\end{equation}
As the solutions must be real, then $b_4=b_5=b_6=0$, and so that $b_1=b_2=b_3=0$. In order that the degeneration be non-trivial, we need $b_7 \neq 0$. By setting $b_7 = 1$ and $a_5 = 0$ we get $a_2=-1$, $a_3=0$, $a_6=1$, $a_7=0$, $a_1=1$ and $a_4=-1$, and thus
$$
X=\diag(1,-1,0,-1,0,1,0)
$$
As $t$ tends to infinite $g_t \cdot \mu \rightarrow \widetilde{\mu}$
$$
\widetilde{\mu}:= \left\{\begin{array}{l}
[e_1, e_2] = e_5,   [e_1, e_3] = e_6,    [e_1, e_4] = 2e_7,  [e_2, e_3] = e_4, {[e_2, e_6]} = e_7,  \\
{[e_3, e_5]} = -e_7
\end{array} \right.
$$
$(\RR^7,\widetilde{\mu})$ is not isomorphic to $(\RR^7,\mu)$ since $\dim \Der (\RR^7,\mu) = 15$ and \newline $\dim \Der (\RR^7,\widetilde{\mu}) = 17$. Therefore the $G_{\phi}$-orbit of $\mu$ is not closed and in consequence $\ggo_{2.2}$ is not an Einstein nilradical.
\end{example}
\begin{center}
\begin{tabular}{|l|c|c|c|c|c|}
\hline
\multicolumn{6}{     |c|      }  {\centering Rank three}  \tabularnewline
\hline
\multicolumn{1}{ |m{1.5cm }|  }  {\centering $\mathfrak{n}$ } &
\multicolumn{1}{  m{0.5cm }|  }  {\centering EN } &
\multicolumn{1}{  m{4.5cm }|  }  {\centering pre-Einstein \\ derivation} &
\multicolumn{1}{  m{1cm}|     }  {\centering Min } &
\multicolumn{1}{  m{0.75cm }| }  {\centering $\dim$ $\Der$} &
\multicolumn{1}{  m{2cm }|    }  {\centering $\dim$ \\ DCS} \tabularnewline
\hline
\hline
$3.1(i_\lambda) $ &\multirow{2}{*}{$\checkmark$}&\multirow{2}{*}{$\frac{1 }{2 }( 1, 1, 1, 2, 2, 2, 3)$}&\multirow{2}{*}{1}&\multirow{2}{*}{$15$}&\multirow{2}{*}{$(4, 1)$}\\
$\lambda \neq 0,1$&&&&&\\
\hline
$3.1(i_\lambda) $ &\multirow{2}{*}{-}&\multirow{2}{*}{$\frac{1 }{2 }( 1, 1, 1, 2, 2,2,3)$}&\multirow{2}{*}{-}&\multirow{2}{*}{$15$}&\multirow{2}{*}{$(4, 1)$}\\
$\lambda =0$& & & & & \\
\hline
$3.1(i_\lambda) $ &\multirow{2}{*}{-}&\multirow{2}{*}{$\frac{1 }{2 }( 1, 1, 1, 2, 2,2,3)$}&\multirow{2}{*}{-}&\multirow{2}{*}{$15$}&\multirow{2}{*}{$(4, 1)$}\\
$\lambda =1$& & & & & \\
\hline
$3.1(iii)$&-&$\frac{ 1}{2 }( 1, 1, 1, 2, 2, 2, 3)$ & - &$ 15$&$(3, 1)$\\
\hline
\end{tabular}
\end{center}
\begin{center}
\begin{tabular}{|l|c|c|c|c|c|}
\hline
\multicolumn{6}{     |c|      }  {\centering Rank three}  \tabularnewline
\hline
\multicolumn{1}{ |m{1.5cm }|  }  {\centering $\mathfrak{n}$ } &
\multicolumn{1}{  m{0.5cm }|  }  {\centering EN } &
\multicolumn{1}{  m{4.5cm }|  }  {\centering pre-Einstein \\ derivation} &
\multicolumn{1}{  m{1cm}|     }  {\centering Min } &
\multicolumn{1}{  m{0.75cm }| }  {\centering $\dim$ $\Der$} &
\multicolumn{1}{  m{2cm }|    }  {\centering $\dim$ \\ DCS} \tabularnewline
\hline
\hline
$3.2$     &$\checkmark$&$\frac{2 }{13 }(1, 5, 6, 6, 7, 7, 8 )$&$1.18$&$17$&$(4, 2, 1)$\\
\hline
$3.3$     &$\checkmark$&$\frac{ 1}{21 }( 5, 12, 15, 17, 27, 22, 27)$ &$0.954$&$15$&$(4, 2, 1)$\\
\hline
$3.4$     &$\checkmark$&$\frac{ 1}{12 }(6, 5, 5, 11, 11, 16, 16 )$&$0.857$&$13$&$(4, 2,)$ \\
\hline
$3.5$     &$\checkmark$&$\frac{1 }{20 }( 10, 7, 11, 17, 21, 24, 28)$&$0.909$&$14 $&$(4, 2)$\\
\hline
$3.6$     &$\checkmark$&$\frac{1 }{13 }( 5, 9, 7, 14, 12, 16, 17)$&$1.18$&$18$&$(4, 1)$\\
\hline
$3.7$     &-&$\frac{1 }{ 3}(1, 2, 2, 2, 3, 4, 4)$      &   -  &$ 15$&$(3, 1)$\\
\hline
$3.8 $    &$\checkmark$&$\frac{ 1}{5 }( 2, 3, 4, 4, 5, 6, 7)$&$1.25$&$19$& $(3, 1)$\\
\hline
$3.9 $    &$\checkmark$&$\frac{ 2}{13 }( 3, 3, 5, 6, 6, 8, 9)$&$1.18$&$ 18$&$(3, 1)$\\
\hline
$3.10$    &$\checkmark$&$\frac{1}{20 }( 12, 7, 11, 16, 19, 23, 30)$&$0.909$&$15 $&$(3, 1)$\\
\hline
$3.11$    &$\checkmark$&$\frac{1}{13}(5, 7, 12, 10, 12, 17, 17)$&$1.18$&$18$&$(3, 1)$\\
\hline
$3.12$    &$\checkmark$&$\frac{ 5}{ 8}( 1, 1, 1, 1, 2, 2, 2)$&$1.33$&$19$&$(3)$\\
\hline
$3.13$    &$\checkmark$&$\frac{1 }{21 }( 8, 11, 15, 15, 19, 27, 30)$&$0.954$&$16$&$(3, 2)$\\
\hline
$3.14 $   &$\checkmark$&$\frac{1 }{13 }(5, 9, 9, 10, 14, 14, 19 )$&$1.18$&$18$&$(3, 1)$\\
\hline
$3.15 $   &$\checkmark$&$\frac{ 1}{11 }(6, 5, 7, 9, 11, 13, 16 )$&$1.10$&$17 $&$(3, 1)$\\
\hline
$3.16 $   &$\checkmark$&$\frac{1 }{12 }( 5, 8, 5, 8, 13, 13, 18)$&$0.857$&$14 $&$(3, 1)$\\
\hline
$3.17 $   &$\checkmark$&$\frac{5 }{21 }(1, 3, 3, 3, 4, 5, 6 )   $&$0.954$&$16$ &$(3, 2, 1)$\\
\hline
$3.18 $   &$\checkmark$&$\frac{2 }{13 }(3, 4, 5, 5, 6, 7, 10 )$&$1.18$&$19$&$(2, 1)$\\
\hline
$3.19 $   &$\checkmark$&$\frac{ 1}{8 }( 5, 6, 6, 5, 6, 11, 11)$&$1.33$&$ 19$&$(2)$\\
\hline
$3.20 $   &$\checkmark$&$\frac{1 }{ 5}(1, 4, 4, 5, 5, 6, 6 )$&1.25&$19 $&$(4, 2)$\\
\hline
$3.21 $   &$\checkmark$&$\frac{1 }{21 }(6, 15, 11, 21, 17, 27, 28 )$&$0.954$&$15$&$(4, 2)$\\
\hline
$3.22$    &$\checkmark$&$\frac{1 }{20 }(7, 12, 10, 19, 17, 29, 24 )$&$0.909$&$15$&(4, 2)\\
\hline
$3.23 $   &$\checkmark$&$\frac{1 }{11 }( 4, 5, 9, 9, 13, 14, 13)$&$1.10$&$17$&$(4, 2)$\\
\hline
$3.24 $   &$\checkmark$&$\frac{ 1}{ 6} ( 5, 3, 4, 4, 8, 7, 7)$ &1.50&$22$& $(3)$\\
\hline
\multicolumn{6}{c}  {\footnotesize{\textsc{Table 3.} Classification of $7$-dimensional indecomposable}}  \tabularnewline
\multicolumn{6}{c}  {\footnotesize{Einstein nilradicals. Rank three case.}}  \tabularnewline
\end{tabular}
\end{center}
\begin{example}\textbf{Applying Nikolayevsky's nice basis criterium} \newline
We consider $\ggo_{3.1(i_{\lambda})}$
$$\left\{
\begin{array}{l}
[e_1,e_2]=e_4, [e_1,e_3]=e_5, [e_1, e_6]=e_7, [e_2, e_3]=e_6, [e_2, e_5]=\lambda e_7, \\
{[e_3, e_4]}=(\lambda - 1) e_7
\end{array}\right.
$$
The basis $\{e_1,...,e_7\}$ is a nice basis to $\ggo_{3.1(i_{\lambda})}$ with any $\lambda$. We can use Theorem \ref{nikonicebasis}. If $\lambda \neq 0,1$, the matrix $U$ is given by
$$
\left( \begin {array}{cccccc} 3&1&1&1&1&-1\\ \noalign{\medskip}1&3&1&
1&-1&1\\ \noalign{\medskip}1&1&3&-1&1&1\\ \noalign{\medskip}1&1&-1&3&1
&1\\ \noalign{\medskip}1&-1&1&1&3&1\\ \noalign{\medskip}-1&1&1&1&1&3
\end {array} \right).
$$
The general solution to the problem $Ux=[1]_6$ is give by
$$
x=\left(t_2,t_1,\frac{1}{2}-t_1-t_2,\frac{1}{2}-t_1-t_2,t_1,t_2\right)^T.
$$
By taking $t_1$ and $t_2$ such that $0 < t_1 < \frac{1}{2}$, $t_2 < \frac{1}{2}-t_1$ we get a solution with positive coordinates. Hence, $\ggo_{3.1(i_{\lambda})}$ with  $\lambda \neq 0,1$ is an Einstein Nilradical.

If $\lambda=0$ the matrix $U$ correspond to $\ggo_{3.1(i_{0})}$ is
$$
\left( \begin {array}{ccccc} 3&1&1&1&-1\\ \noalign{\medskip}1&3&1&1&1
\\ \noalign{\medskip}1&1&3&-1&1\\ \noalign{\medskip}1&1&-1&3&1
\\ \noalign{\medskip}-1&1&1&1&3\end {array} \right).
$$
The general solution to the problem $Ux=[1]_6$ is
$
\left(t, 0 , \frac{1}{2}-t , \frac{1}{2}-t , t\right)^T,
$
and it follows from the nullity of the second coordinate, one obtains $\ggo_{3.1(i_{0})}$ is not an Einstein nilradical. By arguing analogously, one obtains $\ggo_{3.1(i_{1})}$ is not an Einstein nilradical.
\end{example}

\begin{center}
\begin{tabular}{|l|c|c|c|c|c|}
\hline
\multicolumn{6}{     |c|      }  {\centering Rank four }  \tabularnewline
\hline
\multicolumn{1}{ |m{1.5cm }|  }  {\centering $\mathfrak{n}$ } &
\multicolumn{1}{  m{0.5cm }|  }  {\centering EN } &
\multicolumn{1}{  m{4.5cm }|  }  {\centering pre-Einstein \\ derivation} &
\multicolumn{1}{  m{1cm}|     }  {\centering Min } &
\multicolumn{1}{  m{0.75cm }| }  {\centering $\dim$ $\Der$} &
\multicolumn{1}{  m{2cm }|    }  {\centering $\dim$ \\ DCS} \tabularnewline
\hline
\hline
$4.1 $&$\checkmark$&$\frac{ 1}{ 7}( 4, 5, 4, 5, 9, 8, 9)$&$1.4 $&$20$&$(3)$ \\
\hline
$4.2 $&$\checkmark$&$\frac{ 2}{ 5}( 1, 2, 2, 2, 3, 3, 3)$&$1.67$&$25$&$(3)$\\
\hline
$4.3 $&$\checkmark$&$\frac{ 1}{7 }( 5, 5, 6, 5, 4, 10, 9)$&$1.4 $&$21$& $(2)$\\
\hline
$4.4 $&$\checkmark$&$\frac{4 }{5}( 1, 1, 1, 1, 1, 1, 2)$&$1.67$& $28 $&$(1)$\\
\hline
\multicolumn{6}{c}  {\footnotesize{\textsc{Table 4.} Classification of $7$-dimensional indecomposable}}  \tabularnewline
\multicolumn{6}{c}  {\footnotesize{Einstein nilradicals. Rank four case.}}  \tabularnewline
\end{tabular}
\end{center}

\section*{Acknowledgments}

I would like to thank warmly my PhD advisor, professor Jorge Lauret for his dedication and helpful corrections during the writing of this paper.

I would also like to thank the referees for their generous comments on the result of this paper; it is encouraging to receive such positive feedback from the anonymous expert panel.

Finally, I would like to give my special acknowledge to Nadina Rojas and Diego Mej\'ia-Guzm\'an for their enthusiastic support in all possible ways.

The author is supported by a CONICET doctoral fellowship (Argentina).


\end{document}